\documentclass[12pt,reqno]{amsart}

\usepackage{amsfonts}

\usepackage[all]{xy}
\usepackage{amssymb}

\newcommand{\comment}[1]{}

\numberwithin{equation}{section}

\theoremstyle{plain}
\newtheorem{thm}{Theorem}[section]
\newtheorem*{thm*}{Theorem}
\newtheorem{cor}[thm]{Corollary}
\newtheorem{lem}[thm]{Lemma}
\newtheorem{prop}[thm]{Proposition}
\newtheorem*{conj}{Conjecture}

\theoremstyle{definition}

\newtheorem{rem}{Remark}

\def\N{\mathbb N}

\def\Z{\mathbb Z}

\def\eps{\varepsilon}

\def\bydef{\overset{\mathrm {def}}=}

\def\gl{\mathfrak{gl}}
\def\sl{\mathfrak{sl}}
\def\so{\mathfrak{so}}
\def\sp{\mathfrak{sp}}

\def\LR{\mathrm N}

\DeclareMathOperator{\End}{End} %

\def\im{\operatorname{im}}
\def\Tr{\operatorname{Tr}}
\def\soc{\operatorname{soc}}

\def\a{\alpha}

\def\1{{\mathbf 1}}
\def\0{{\mathbf 0}}
\def\o{\otimes}

\def\g{\mathfrak g}
\def\b{\mathfrak b}
\def\h{\mathfrak h}
\def\n{\mathfrak n}

\def\glV_#1{\boldsymbol\Gamma_{#1}}
\def\spV_#1{\boldsymbol\Gamma_{\<#1\>}}
\def\soV_#1{\boldsymbol\Gamma_{[#1]}}

\def\FF{{\mathbf F}}

\def\SS{\mathbb S}
\def\VV{\mathbf V}
\def\TT{\mathrm T}

\def\cplus{\hbox{$\subset${\raise0.3ex\hbox{\kern -0.55em ${\scriptscriptstyle +}$}}\ }}

\def\<{\langle}
\def\>{\rangle}
\def\({\left(}
\def\){\right)}
\def\[{\left[}
\def\]{\right]}

\begin{document}

\title[Tensor representations of classical locally finite Lie
algebras] {Tensor representations of classical \\ locally finite Lie
algebras }

\author{Ivan Penkov}
\author{Konstantin Styrkas}

\address{Jacobs University Bremen, Bremen 28759, Germany}

\maketitle


\begin{abstract}

The structure of tensor representations of the classical
finite-dimensional Lie algebras was described by H. Weyl. In this
paper we extend Weyl's results to the classical infinite-dimensional
locally finite Lie algebras $\gl_\infty$, $\sl_\infty$, $\sp_\infty$
and $\so_\infty$, and study important new features specific to the
infinite-dimensional setting.

Let $\g$ be one of the above locally finite Lie algebras and let $\VV$ be the
natural representation of $\g$. The tensor representations of $\g$
have the form $\VV^{\o d}$ for the cases $\g =
\sp_\infty,\so_\infty$, and the form $\VV^{\o p}\o \VV_*^{\o q}$ for
the cases $\g = \gl_\infty,\sl_\infty$, where $\VV_*$ is the
restricted dual of $\VV$. In contrast with the finite-dimensional
case, these tensor representations are not semisimple. We explicitly
describe their Jordan-H\"older constituents, socle filtrations, and
indecomposable direct summands.
\end{abstract}




\setcounter{section}{-1}

\section{Introduction}

One of H. Weyl's beautiful results is the description of the
structure of the space of tensors of given rank as a module over the
various classical groups of transformations of the underlying space. In this
paper we extend Weyl's constructions to the locally-finite
Lie algebras $\gl_\infty, \sl_\infty, \so_\infty$ and $\sp_\infty$,
and highlight some important differences and new features arising
in the infinite-dimensional case.

Each of the Lie algebras $\gl_\infty$, $\sl_\infty, \so_\infty$ and
$\sp_\infty$ can be defined as the union of the respective series of
finite-dimensional Lie algebras under the obvious inclusions. We refer to
$\gl_\infty$, $\sl_\infty, \so_\infty$ and $\sp_\infty$ as the classical infinite-dimensional locally finite Lie algebras.
The structure of these Lie algebras and their representations have been studied actively in recent years; an incomplete list of references for this field is
\cite{Ba,BB,BS,DP1,DP2,DPS,DPW,N,NS,Na,NP,O,PS,PZ}. Nevertheless,
the structure of the tensor representations of the classical infinite-dimensional locally finite Lie algebras has not yet been sufficiently explored.

Recall that for the simple classical Lie groups $GL_n, SO_n$ and
$Sp_{2n}$ every irreducible finite-dimensional module can be
realized inside the covariant tensor algebra of the natural
representation $V$. In particular, for the dual representation $V^*$
we have $V^* \cong V$ in the cases of $SO_n$ and $SP_{2n}$, and $V^*
\cong \bigwedge^{n-1}V$ in the case of $SL_n$. However, for the
reductive Lie algebra $GL_n$ not every finite-dimensional
irreducible representation occurs in the tensor algebra $\TT(V)$.
For example, the dual representation $V^*$ does not appear in $\TT(V)$.
Thus, in order to obtain all irreducible finite-dimensional
$GL_n$-modules, we must consider mixed tensors, i.e. the tensor
algebra $\TT(V \oplus V^*)$.

A similar phenomenon occurs for the simple Lie algebra $\sl_\infty$.
Let $\VV$ denote the natural representation of $\sl_\infty$. One can
show that the (restricted) dual representation $\VV_*$ does not
occur in the tensor algebra $\TT(\VV)$, and we study the space of
mixed tensors $\TT(\VV \oplus \VV_*)$. The main new feature of these
tensor representations is their failure to be completely reducible.
For $\so_\infty$ and $\sp_\infty$ the natural representation $\VV$
is self-dual, and we consider just the tensor algebra $\TT(\VV)$,
which also fails to be a completely reducible $\so_\infty$- or
$\sp_\infty$-module.

The main purpose of this paper is to describe the structure of the
tensor algebras $\TT(\VV\oplus \VV_*)$ and $\TT(\VV)$
as explicitly as possible. First, we decompose the tensor algebra
into indecomposable submodules of finite length. The answer is very
simple: for $\g = \gl_\infty, \sl_\infty$ the indecomposables are of
the form $\glV_{\lambda;0} \o \glV_{0;\mu}$, where
$\glV_{\lambda;0}$ and $\glV_{0;\mu}$ are the simple irreducible
$\sl_\infty$-submodules of the tensor algebras $\TT(\VV)$ and
$\TT(\VV_*)$ respectively, and for $\g = \so_\infty, \sp_\infty$ the
indecomposable modules are the $\gl_\infty$-modules
$\glV_{\lambda;0}$, considered as $\g$-modules.

We then explicitly describe the socle filtration of
$\glV_{\lambda;0} \o \glV_{0;\mu}$ as an $\sl_\infty$-module, and
the socle filtrations of $\glV_{\lambda;0}$ both as an
$\so_{\infty}$- and an $\sp_\infty$-module. In particular, we prove
that each indecomposable module has a simple socle.

In the last section of the paper we extend our description of tensor representations to a wider class of Lie algebras $\g = \mathfrak k \cplus \mathfrak m$, where the ideal $\mathfrak k$ is a classical simple locally-finite Lie algebra, and $\mathfrak m$ satisfies a certain technical condition. This class of Lie algebras includes, in particular, the root-reductive Lie algebras. The main result in this case states that the natural representation $\VV$ of $\mathfrak k$ can be equipped with a $\g$-module structure, and the socle filtration of any tensor representation of $\g$ coincides with its socle filtration as a $\mathfrak k$-module.

\begin{conj}
Each indecomposable direct summand of a tensor representations of
$\gl_\infty, \sl_\infty, \sp_\infty$ and $\so_\infty$ is rigid, i.e. its socle
filtration coincides with its radical filtration.
\end{conj}
The examples computed in this paper imply that this conjecture holds
for tensor representations of rank $\le 4$.

\section{Preliminaries}

By $\N$ we denote the set of positive integers.
In this paper we work with a field $\Bbbk$ of characteristic zero.
All vector spaces, Lie algebras, and their modules are assumed to be defined over
$\Bbbk$. Tensor products are also over $\Bbbk$. All algebras and
modules are assumed to be at most countable dimensional over
$\Bbbk$. The sign $\cplus$ stands for semidirect sum of Lie algebras: in
$\mathfrak k \cplus \mathfrak m$ the Lie algebra $\mathfrak k$ is an ideal.

\medskip

Let $V$ be a finite length module over some algebra. The {\it socle}
of $V$, denoted $\soc V$, is the maximal semisimple submodule of $V$. Equivalently, $\soc V$ is the sum of all simple submodules of
$V$. The {\it socle filtration} of $V$ is defined inductively by
\begin{equation*}
    \soc^{(0)} V = 0,    \qquad\qquad
    \soc^{(i+1)} V/\soc^{(i)} V = \soc(V/\soc^{(i)} V), \quad i\ge 0.
\end{equation*}
The smallest positive integer $l$ such that $\soc^{(l)} V = V$ is
called the {\it Loewy length} of $V$. The semisimple modules
$\overline{\soc}^{(i)} V \overset{\rm def}{=} \soc^{(i)}
V/\soc^{(i-1)} V$ are called the {\it layers} of the socle
filtration.

\medskip

A \textit{partition} $\lambda$ is by definition a finite nonstrictly
decreasing sequence of positive integers:
\begin{equation*}
    \lambda = (\lambda_1 \ge \lambda_2 \ge \dots \ge \lambda_k).
\end{equation*}
For any partition $\lambda$ we set $|\lambda| \bydef \lambda_1 +
\dots + \lambda_k$, and use the convention that $\lambda_i=0$ if
$i>k$. The empty partition is denoted by $0$.

For any partition $\lambda$ we denote by $2\lambda$ the partition
determined by $(2\lambda)_i = 2\lambda_i$ for all $i$, and by
$\lambda^\top$ the partition determined by $(\lambda^\top)_i = \#
\{j \,|\, \lambda_j \ge i\}$ for all $i$. The Young diagram for
$\lambda^\top$ is obtained by transposing the Young diagram for
$\lambda$.

\medskip

Irreducible representations of the symmetric group $\mathfrak S_d$
are parameterized by partitions $\lambda$ satisfying $|\lambda| =
d$, and can be realized inside the regular representation
$\Bbbk[\mathfrak S_d]$. We denote
\begin{equation*}
    H_{\lambda} \ \bydef \ \Bbbk[\mathfrak S_d] \, c_\lambda,
\end{equation*}
where $c_\lambda \in \Bbbk[\mathfrak S_d]$ is the Young projector
corresponding to the standard Young tableau of shape $\lambda$ (see
e.g. \cite{FH} for details). The left regular action makes
$H_\lambda$ an irreducible $\mathfrak S_d$-module, and any
irreducible representation of $\mathfrak S_d$ can be obtained in
this way.

\medskip

For any vector space $V$, the symmetric group $\mathfrak S_d$ acts
in $V^{\o d}$ by permuting the tensor factors. For
any partition $\lambda$ with $|\lambda| = d$ we denote
\begin{equation*}
    \mathbb S_\lambda V = \im \biggr(c_\lambda: V^{\o d} \to V^{\o d} \biggr)
\end{equation*}
The correspondence $V \rightsquigarrow \mathbb S_\lambda V$ is
called the {\it Schur functor} corresponding to $\lambda$.

\medskip

The \textit{Littlewood-Richardson coefficients}
$\LR_{\lambda,\mu}^\nu$ are nonnegative integers, determined for any
partitions $\lambda,\mu,\nu$ by the relation $S_\lambda \, S_\mu =
\sum_\nu \LR_{\lambda,\mu}^\nu \, S_\nu$, where $S_\lambda$ denotes
the Schur symmetric polynomial corresponding to the partition
$\lambda$. In particular, $\LR_{\lambda,\mu}^\nu=0$ unless $|\nu| =
|\lambda|+|\mu|$.

\section{Tensor representations of $\gl_\infty$ and $\sl_\infty$}
\label{sec:gl}


Let $\VV,\VV_*$ be countable dimensional vector spaces, and let
$\<\cdot,\cdot\>: \VV \o \VV_* \to \Bbbk$ be a non-degenerate
pairing. The Lie algebra
$\gl_\infty$ is defined as the space $\VV \o \VV_*$, equipped with
the Lie bracket
\begin{equation}\label{eq:def gl Lie bracket}
    [u \o u^*, v \o v^*] =
    \<u^*,v\> \, u \o v^* - \<v^*,u\> \, v \o u^*,
    \qquad\qquad
    u,v \in \VV, \ u^*,v^* \in \VV_*.
\end{equation}
The kernel of the map $\<\cdot,\cdot\>$ is a Lie subalgebra of $\gl_\infty$,
which we denote $\sl_\infty$.


As observed by G. Mackey \cite{M}, there always exist dual bases
$\{\xi_i\}_{i \in \mathfrak I}$ of $\VV$ and
$\{\xi_i^*\}_{i\in\mathfrak I}$ of $\VV_*$, indexed by a countable
set $\mathfrak I$, so that we have $\<\xi_j^*, \xi_i\> =
\delta_{i,j}$ for $i,j \in \mathfrak I$. This gives another, more
straightforward coordinate definition of $\gl_\infty$ as the Lie
algebra with a linear basis $\{E_{i,j}= \xi_i \o \xi_j^*\}_{i,j\in
\mathfrak I}$ satisfying the usual commutation relations
$[E_{i,j},E_{k,l}] = \delta_{j,k} E_{i,l} - \delta_{i,l} E_{k,j}$.

We call $\VV$ the {\it natural representation} of $\gl_\infty$ and
$\sl_\infty$, and $\VV_*$ its {\it restricted dual}. For any
nonnegative integers $p,q$ we define the {\it tensor representation}
$\VV^{\o(p,q)}$ as the vector space $\VV^{\o p} \otimes \VV_*^{\o q}$, equipped with the following $\gl_\infty$-module structure:
\begin{multline*}\label{eq:tensor action gl}
    (u \o u^*) \cdot ( v_1 \o\dots\o v_p \o v^*_1 \o\dots\o v^*_q)
    \\
    = \sum_{i=1}^p \<u^*,v_i\> \, v_1 \o \dots \o v_{i-1} \o
    u \o v_{i+1} \o \dots\o v_p \o v^*_1 \o\dots \o v_q^*
    \\
    - \sum_{j=1}^q \<v_j^*,u\> \, v_1 \o \dots \o v_p \o v^*_1 \o\dots\o v_{j-1}^* \o u^* \o v_{j+1}\o \dots \o
    v_q^*
\end{multline*}
for $u,v_1,\dots,v_p \in V$ and $u^*,v_1^*,\dots,v_q^* \in \VV_*$. The product of symmetric groups $\mathfrak S_p \times \mathfrak S_q$
acts in $\VV^{\o(p,q)}$ by permuting the factors, and this action
commutes with the action of $\gl_\infty$. We express this by saying
that $\VV^{\o(p,q)}$ is a $(\gl_\infty,\mathfrak S_p \times
\mathfrak S_q)$-module, and use similar notation throughout the
paper.

\medskip

Our main goal in this section is to reveal the structure of the tensor
representations as a $\gl_\infty$-module, and in particular to
identify the Jordan-H\"older constituents of $\VV^{\o(p,q)}$. We describe these modules explicitly as appropriately defined highest weight $\gl_\infty$-modules. Consider the direct sum
decomposition $\gl_\infty = \h_\gl \oplus \( \bigoplus_{\a \in
\Delta_{\gl}} \, \Bbbk \, X_\a^{\gl} \)$, where
\begin{equation*}
    \h_\gl = \bigoplus_{i \in \mathfrak I} \ \Bbbk \, E_{i,i},
    \qquad\qquad
    \Delta_{\gl} = \left\{ \eps_i-\eps_j \, \biggr|\, i,j \in \mathfrak I,\ i \ne j
    \right\},
    \qquad\qquad
    X^{\gl}_{\eps_i-\eps_j} = E_{i,j}\ ,
\end{equation*}
and $\eps_i$ denotes the functional on $\h_\gl$ determined by
$\eps_i(E_{j,j}) = \delta_{i,j}$. We have $[H,X_\a^\gl] = \a(H)
X_\a^{\gl}$ for any  $\a\in \Delta_\gl$ and $H \in \h_\gl$. As
usual, we refer to elements of $\Delta_\gl$ as {\it roots}, and to
functionals on $\h_\gl$ as {\it weights}.

\begin{rem}
For an algebraically closed field $\Bbbk$ the general notion of a
Cartan subalgebra of $\gl_\infty$ has been defined and studied in
\cite{NP,DPS}.  The subalgebra $\h_\gl$ is an example of a {\it
splitting Cartan} subalgebra, and all splitting Cartan subalgebras
of $\gl_\infty$ are conjugated.
\end{rem}

From now on we identify the index set $\mathfrak I$ with
$\Z\setminus\{0\}$, and consider the polarization of the root system
$\Delta_\gl = \Delta_\gl^+ \coprod -\Delta_\gl^+$, where the set
$\Delta_\gl^+$ of {\it positive roots} is given by
\begin{equation*}
    \Delta_\gl^+ = \{ \eps_i-\eps_j \,|\, 0<i < j \} \,\bigcup \,
    \{ \eps_i-\eps_j \,|\, i < j < 0 \}
    \,\bigcup\, \{ \eps_i-\eps_j \,|\, j < 0 < i \}.
\end{equation*}
We denote $\n_\gl = \bigoplus_{\a\in\Delta_\gl^+} \g_\a$ and define $\b_\gl$ as the Lie subalgebra of $\gl_\infty$, generated by $\n_\gl$ and $\h_\gl$. It is clear that $\n_\gl$ is a Lie subalgebra of $\gl_\infty$,  that $\b_\gl = \n_\gl \cplus \h_\gl$, and that $[\b_\gl,\b_\gl] =
\n_\gl$.

Let $V$ be a $\gl_\infty$-module, and let $v\in V$. We say that $v$
is a {\it highest weight vector} if it generates a one-dimensional
$\b_\gl$-module. Any such $v$ must satisfy
\begin{equation*}
    \n_\gl \, v = 0,
    \qquad\qquad\quad
    H \, v = \chi(H)\, v
    \qquad
    \forall H \in \h_\gl,
\end{equation*}
for some $\chi \in \h_\gl^*$. We say that $V$ is a {\it highest weight
module} if $V$ is generated by a highest weight vector $v$ as above;
the functional $\omega$ is then called the {\it highest weight of $V$}. As
in the finite-dimensional case, it is easy to prove that for each
$\chi \in \h_\gl^*$ there exists a unique irreducible highest weight
$\gl_\infty$-module with highest weight $\omega$.

\begin{rem}
The subalgebra $\b_\gl$ is an example of a {\it splitting Borel}
subalgebra of $\gl_\infty$, see \cite{DP2} for a classification of
the splitting Borel subalgebras in $\gl_\infty$ over algebraically
closed fields. In this paper we only consider the notion of a highest
weight module, associated with $\b_\gl$. This choice has the important property
that all Jordan-H\"older constituents of all tensor representations
are highest weight modules.
\end{rem}

In particular, the natural representation $\VV$ is a highest weight
module with highest weight $\eps_1$, generated by the highest weight
vector $\xi_1$; similarly, $\VV_*$ is a highest weight vector of
weight $-\eps_{-1}$, generated by the highest weight vector $\xi_{-1}^*$.

\medskip

We now describe the Weyl construction for $\gl_\infty$. For any pair
of indices $I = (i,j)$ with $i\in\{1,2,\dots,p\}$ and
$j\in\{1,2,\dots,q\}$, define the contraction
\begin{equation*}
\begin{gathered}
    \Phi_{I}: \VV^{\o (p,q)} \to \VV^{\o (p-1,q-1)},\\
    v_1 \o\dots\o v_p \o v^*_1 \o\dots\o v^*_q \mapsto \<v_{j}^*,v_{i}\> \, v_1 \o \dots \o \hat
    v_{i} \o \dots\o v_p \o v^*_1 \o\dots\o \hat v_{j}^* \o \dots \o v_q^*,
\end{gathered}
\end{equation*}
and consider the $(\gl_\infty, \mathfrak S_p \times \mathfrak
S_q)$-submodule $\VV^{\{p,q\}}$ of $\VV^{\o (p,q)}$,
\begin{equation*}
    \VV^{\{p,q\}} \bydef \bigcap_{I} \ker
    \biggr( \Phi_{I}: \VV^{\o(p,q)} \to
    \VV^{\o(p-1,q-1)}\biggr).
\end{equation*}
Set also $\VV^{\{p,0\}} \bydef \VV^{\o p}$ and $\VV^{\{0,q\}}
\bydef \VV_*^{\o q}$. For any partitions $\lambda,\mu$ such that
$|\lambda|=p$ and $|\mu|=q$, define the $\gl_\infty$-submodule of
$\VV^{\o(p,q)}$
\begin{equation*}
    \glV_{\lambda;\mu} \bydef  \VV^{ \{p,q\} } \cap (\SS_\lambda \VV \o \SS_\mu \VV_*).
\end{equation*}

\medskip

\begin{thm}\label{thm:mixed Schur-Weyl infinite}
For any $p,q$ there is an isomorphism of $(\gl_\infty, \mathfrak S_p
\times \mathfrak S_q)$-modules
\begin{equation}\label{eq:mixed Schur-Weyl infinite}
    \VV^{\{p,q\}} \cong \bigoplus_{|\lambda|=p} \bigoplus_{|\mu| = q}
    \glV_{\lambda;\mu} \o (H_\lambda \o H_\mu).
\end{equation}
For any partitions $\lambda,\mu$, the $\gl_\infty$-module
$\glV_{\lambda;\mu}$ is an irreducible highest weight module with
highest weight $\chi \bydef \sum_{i \in \N} \lambda_i \, \eps_i - \sum_{i
\in \N} \mu_i \, \eps_{-i}$. Furthermore, $\glV_{\lambda;\mu}$ is
irreducible when regarded by restriction as an $\sl_\infty$-module.
\end{thm}

\begin{proof}
Pick an enumeration $\mathfrak I = \{k_1,k_2,k_3,\dots\}$ of the
index set $\mathfrak I = \Z\setminus \{0\}$, such that $k_i = i$ for
$1\le i \le p$, and $k_{p+i} = -i$ for $1\le i\le q$. For each $n$
denote by $V_n$ the subspace of $\VV$ spanned by
$\xi_{k_1},\dots,\xi_{k_n}$, and by $V_n^*$ the subspace of
$\VV_*$ spanned by $\xi_{k_1}^*,\dots,\xi_{k_n}^*$; the pairing
between $\VV$ and $\VV_*$ restricts to a nondegenerate pairing
between $V_n$ and $V_n^*$.

Denote by $\g_n$ the Lie subalgebra of $\gl_\infty$ generated by
$\{E_{k_i,k_j}\}_{1\le i,j\le n}$, and set $\b_n \bydef \b_\gl \cap
\g_n$, $\h_n \bydef \h_\gl \cap \g_n$. It is clear that $\g_n \cong
\gl_n$, and that $\h_n$ (respectively, $\b_n$) is a Cartan (respectively, Borel) subalgebra of
$\g_n$. Moreover, the inclusion $\g_n \hookrightarrow \g_{n+1}$
restricts to inclusions $\h_n \hookrightarrow \h_{n+1}$ and $\b_n \hookrightarrow \b_{n+1}$.

The tensor representations $(V_n)^{\o(p,q)}$ and the contractions
$\Phi_I^{(n)}: (V_n)^{\o(p,q)} \to (V_n)^{\o(p-1,q-1)}$ are defined
by analogy with their infinite-dimensional counterparts. We set
\begin{equation*}
    (V_n)^{\{p,q\}} \bydef \bigcap_{I} \ker \Phi_{I}^{(n)},
    \qquad\qquad
    \varGamma_{\lambda;\mu}^{(n)} \bydef  (V_n)^{ \{p,q\} } \cap (\SS_\lambda V_n \o \SS_\mu V_n^*).
\end{equation*}

A version of the finite-dimensional Weyl construction (see Appendix
for more details) implies that for $n\ge p+q$ the $\gl_n$-module
$\varGamma_{\lambda;\mu}^{(n)}$ is an irreducible highest weight
module with highest weight $\chi$, regarded by restriction as a
functional on $\h_n$. Futhermore, we have $(\g_n, \mathfrak S_p \times
\mathfrak S_q)$-module isomorphisms
\begin{equation}\label{eq:mixed Schur-Weyl finite}
    (V_n)^{\{p,q\}} \cong \bigoplus_{|\lambda|=p} \bigoplus_{|\mu| = q}
    \varGamma_{\lambda;\mu}^{(n)} \o (H_\lambda \o H_\mu).
\end{equation}
It is clear from the definitions that we have the following diagram
of inclusions:
\begin{equation}\label{eq:exhaustions gl}
\begin{gathered}\xymatrix@C=0pt@R=0pt{
    (V_{p+q})^{\{p,q\}} & \subset & (V_{p+q+1})^{\{p,q\}} & \subset &\dots &\subset &(V_n)^{\{p,q\}}
    &\subset &\dots &\subset  &\bigcup_{n\in\N} (V_n)^{ \{p,q\} } & = & \VV^{\{p,q\}}
    \\
    \cup &&      \cup &&     &&     \cup &&     &&     \cup &&    \cup
    \\
    \varGamma_{\lambda;\mu}^{(p+q)} & \subset & \varGamma_{\lambda;\mu}^{(p+q+1)} &\subset &\dots &\subset &\varGamma_{\lambda;\mu}^{(n)}
    &\subset &\dots &\subset &\bigcup_{n\in\N} \varGamma_{\lambda;\mu}^{(n)} & = & \glV_{\lambda;\mu}
    }
\end{gathered}.
\end{equation}
The irreducibility of the $\gl_n$-module
$\varGamma_{\lambda;\mu}^{(n)}$ for each $n$ implies the irreducibility of the
$\gl_\infty$-module $\glV_{\lambda;\mu}$. Similarly, the isomorphism
\eqref{eq:mixed Schur-Weyl infinite} follows from the isomorphisms
\eqref{eq:mixed Schur-Weyl finite}, and it remains to show that
$\glV_{\lambda;\mu}$ is a highest weight $\gl_\infty$-module with
highest weight $\chi$.

For each $n\ge p+q$ the highest weight subspace
$\varGamma_{\lambda;\mu}^{(n)}[\chi]$ is one-dimensional, and
therefore
\begin{equation*}
    \varGamma_{\lambda;\mu}^{(p+q)}[\chi] = \varGamma_{\lambda;\mu}^{(p+q+1)}[\chi]=
    \dots = \varGamma_{\lambda;\mu}^{(n)}[\chi] = \dots = \bigcup_{n\in\N} \varGamma_{\lambda;\mu}^{(n)}[\chi] = \glV_{\lambda;\mu}[\chi].
\end{equation*}
We also have $\b_\gl \cdot \glV_{\lambda;\mu}[\chi] =
\bigcup_{n\in\N} \( \b_n \cdot \varGamma_{\lambda;\mu}^{(n)}[\chi]
\) = \bigcup_{n\in\N} \varGamma_{\lambda;\mu}^{(n)}[\chi] =
\glV_{\lambda;\mu}[\chi]$, which means that any nonzero $v
\in\glV_{\lambda;\mu}[\chi]$ is a highest weight vector generating
the $\gl_\infty$-module $\glV_{\lambda;\mu}[\chi]$. This completes
the proof of the theorem.
\end{proof}

\medskip

Applying Theorem \ref{thm:mixed Schur-Weyl infinite} to the special
cases $p=0$ and $q=0$, we obtain a version of Schur-Weyl duality for
$\gl_\infty$:
\begin{equation}\label{eq:Schur-Weyl infinite}
    \VV^{\o p} \cong \bigoplus_{|\lambda| = p}
    \glV_{\lambda;0} \o H_\lambda,
    \qquad\qquad
    \VV_*^{\o q} \cong \bigoplus_{|\mu| = q}
    \glV_{0;\mu} \o H_\mu.
\end{equation}


Next, we study the structure of the $\gl_\infty$-module
$\VV^{\o(p,q)}$. The main feature and the crucial difference from
the finite-dimensional case is that $\VV^{\o(p,q)}$ fails to be
completely reducible.

\begin{thm}\label{thm:socle series gl}
Let $p,q$ be nonnegative integers, and let $\ell = \min(p,q)$. Then
the Loewy length of the $\gl_\infty$-module $\VV^{\o (p,q)}$ equals
$\ell+1$, and
\begin{equation}\label{eq:socle series gl}
    \soc^{(r)} \VV^{\o(p,q)} = \bigcap_{I_1,\dots,I_r} \ker
    \biggr( \Phi_{I_1,\dots,I_r}: \VV^{\o(p,q)} \to
    \VV^{\o(p-r,q-r)}\biggr),
    \qquad
    r=1,\dots,\ell.
\end{equation}
Moreover, the socle filtration of $\VV^{\o (p,q)}$, regarded as an
$\sl_\infty$-module, coincides with \eqref{eq:socle series gl}.
\end{thm}

\begin{proof}
Denote by $\FF^{(r)}$ the subspaces of $\VV^{\o(p,q)}$ on the right
side of \eqref{eq:socle series gl}, and for each $n$ set
\begin{equation*}
    F_n^{(r)} = \bigcap_{I_1,\dots,I_r} \ker
    \biggr( \Phi_{ \{I_1,\dots,I_r\} }: (V_n)^{\o(p,q)} \to
    (V_n)^{\o(p-r,q-r)}\biggr),
    \qquad r=1,\dots,\ell,
\end{equation*}
where $\{V_n\}_{n \in \N}$ is the exhaustion of $\VV$ used in the
proof of Theorem \ref{thm:mixed Schur-Weyl infinite}. Then we have
the commutative diagram of $(\gl_\infty,\mathfrak S_p \times
\mathfrak S_q)$-modules
\begin{equation*}
\begin{gathered}
\xymatrix@C=5pt{
    \ar@{.>}[d] & &  \ar@{.>}[d] & & & & \ar@{.>}[d] & & \ar@{.>}[d]
    \\
    F_n^{(1)}  \ar@{->}[d] & \subset & F_n^{(2)}  \ar@{->}[d] & \subset & \dots & \subset &
    F_n^{(\ell)}  \ar@{->}[d] & \subset & (V_n)^{\o(p,q)} \ar[d]
    \\
    F_{n+1}^{(1)} \ar@{.>}[d] & \subset
    & F_{n+1}^{(2)} \ar@{.>}[d] & \subset & \dots & \subset & F_{n+1}^{(\ell)} \ar@{.>}[d]
    & \subset &     (V_{n+1})^{\o(p,q)} \ar@{.>}[d]
    \\
    \FF^{(1)}  & \subset
    & \FF^{(2)} & \subset & \dots & \subset & \FF^{(\ell)}  & \subset & \VV^{\o(p,q)}
    }
\end{gathered},
\end{equation*}
in which the vertical arrows are obtained as restrictions of the
inclusions in the rightmost column, and yield exhaustions of
$\FF^{(r)}$. Denote for convenience  $\FF^{(0)} = 0$,
$\FF^{(\ell+1)} = \VV^{\o(p,q)}$, and similarly $F_n^{(0)} = 0$,
$F_n^{(\ell+1)} = (V_n)^{\o(p,q)}$ for all $n$. It is easy to check
that the induced maps $F_n^{(r+1)}/ F_{n}^{(r)} \to F_{n+1}^{(r+1)}/
F_{n+1}^{(r)}$ are injective for all $n,r$, and therefore for each
$r=0,1,\dots,\ell$ the layer $\FF^{(r+1)}/\FF^{(r)}$ is the union of
quotients $F_n^{(r+1)}/ F_{n}^{(r)}$.

It is a standard exercise (see Appendix for more details) to show
that for each $r$ there exists a $(\gl_n,\mathfrak S_p \times
\mathfrak S_q)$-module isomorphism
\begin{equation}\label{eq:middle layer finite}
    F_n^{(r+1)}/ F_{n}^{(r)} \cong \bigoplus_{|\lambda|=p-r} \bigoplus_{|\mu|=q-r}
    \varGamma_{\lambda;\mu}^{(n)} \o H(\lambda,\mu;r),
\end{equation}
where $H(\lambda,\mu;r)$ are some $\mathfrak S_p \times \mathfrak
S_q$-modules. As in the proof of Theorem \ref{thm:mixed Schur-Weyl
infinite}, it follows that
\begin{equation*}
    \FF^{(r+1)}/\FF^{(r)} \cong \bigoplus_{|\lambda|=p-r} \bigoplus_{|\mu|=q-r}
    \glV_{\lambda;\mu} \o H(\lambda,\mu;r).
\end{equation*}
In particular, this shows that $\VV^{\o(p,q)}$ has a finite
Jordan-H\"older series with irreducible constituents of the form $\glV_{\lambda;\mu}$ for appropriate $\lambda,\mu$. Moreover,
$\FF^{(r)}$ is characterized as the unique submodule of
$\VV^{\o(p,q)}$ such that for any $\lambda,\mu$
\begin{equation*}
    [\FF^{(r)}:\glV_{\lambda;\mu}] =
    \begin{cases}
    [\VV^{\o(p,q)}:\glV_{\lambda;\mu}] & \text{ if }|\lambda|>p-r \text{ and }
    |\mu|>q-r,\\
    0 &\text{ otherwise }.
    \end{cases}
\end{equation*}

We use induction on $r$ to prove the main statement of the theorem:
$\FF^{(r)} = \soc^{(r)} \VV^{\o(p,q)}$. The base of induction $r=0$
is trivial since $\FF^{(0)} = \soc^{(0)} \VV^{\o(p,q)} = 0$.

Suppose that $\FF^{(r)} = \soc^{(r)} \VV^{\o(p,q)}$ for some $r$.
The quotient $\FF^{(r+1)}/\FF^{(r)}$ is a semisimple
$\gl_\infty$-module, hence $\FF^{(r+1)} \subset \soc^{(r+1)}
\VV^{\o(p,q)}$. Now let $\mathbf U$ be a simple submodule of
$\VV^{\o(p,q)}/\FF^{(r)}$; then $\mathbf U \cong \glV_{\lambda;\mu}$
with $\lambda,\mu$ satisfying $|\lambda| = p-s$ and $|\mu| = q - s$
for some $s\ge r$. In particular $\mathbf U \subset
\FF^{(s+1)}/\FF^{(r)}$. Our goal is to show that in fact $s=r$;
indeed, we would then conclude that $\soc^{(r+1)} \VV^{\o(p,q)}
\subset \FF^{(r+1)}$, proving the induction step.

Fix a vector $u \in \VV^{\o(p,q)}$ of weight $\chi = \sum_i
\lambda_i \eps_i - \sum_i \mu_i \eps_{-i}$, such that the image of
$u$ under the projection $\VV^{\o(p,q)} \to \VV^{\o(p,q)}/\FF^{(r)}$
is generates $\mathbf U$. Fix a large enough $m \in \N$ such that $u
\in (V_m)^{\o(p,q)}$. We may further assume without loss of
generality that $u$ generates a $\gl_m$-submodule of
$(V_m)^{(\o(p,q)}$ isomorphic to $\varGamma_{\lambda;\mu}^{(m)}$.

For each $n>m$ denote by $\pi_n: F^{(s+1)}_n \to
\varGamma_{\lambda;\mu}^{(n)} \o H(\lambda,\mu;s)$ the projection
corresponding to the decomposition \eqref{eq:middle layer finite}.
The map $\pi_n$ can be described explicitly, and we show in Proposition \ref{thm:proval} in the Appendix that  $u - \pi_n(u) \notin F^{(s-1)}_n$ for infinitely many $n$. On the other hand, $u$ generates a submodule of $ F^{(s+1)}_n/ F^{(r)}_n$ isomorphic to $\varGamma_{\lambda;\mu}^{(n)}$, which implies that $u-\pi_n(u)
\in F^{(r)}_n$ for all $n$. We conclude that $s\le r$, which
implies that in fact $s=r$. The induction is now complete, and thus
$\FF^{(r)} = \soc^{(r)} \VV^{\o(p,q)}$ for all $r$.

Finally, we need to verify that the filtration $\{\FF^{(r)}\}$ is
also the socle filtration of $\VV^{\o(p,q)}$ regarded as an
$\sl_\infty$-module. The irreducibility of $\glV_{\lambda;\mu}$ as an
$\sl_\infty$-module implies that the layers of the filtration
$\FF^{(r)}$ remain semisimple. Therefore $\FF^{(r)} \subset
\soc^{(r)} \VV^{\o(p,q)}$ for all $r$. The proof of the opposite
inclusion, given above for $\gl_\infty$, works without any
alterations for $\sl_\infty$ as well, and this completes the proof
of the theorem.
\end{proof}

\bigskip

Next, we describe the indecomposable constituents of the tensor
representations of $\gl_\infty$.

\begin{thm}\label{thm:indecomposables gl}
For any partitions $\lambda,\mu$ the $\gl_\infty$-module
$\glV_{\lambda; 0} \o \glV_{0;\mu}$ is indecomposable, and
\begin{equation}\label{eq:socle layers gl}
    \overline\soc^{(r+1)} (\glV_{\lambda; 0} \o \glV_{0;\mu}) \cong
    \bigoplus_{\lambda',\mu'}
    \( \sum_{|\gamma| = r} \LR_{\lambda',\gamma}^\lambda \LR_{\mu',\gamma}^\mu \) \
    \glV_{\lambda';\mu'}.
\end{equation}
The same applies to $\glV_{\lambda; 0} \o \glV_{0;\mu}$ regarded as an $\sl_\infty$-module.
\end{thm}

\begin{proof}
Schur-Weyl duality \eqref{eq:Schur-Weyl infinite} for $\gl_\infty$
implies that
\begin{equation*}
    \VV^{\o(p,q)} \cong \bigoplus_{|\lambda|=p} \bigoplus_{|\mu| = q}
    (\glV_{\lambda;0} \o \glV_{0;\mu} )\o (H_\lambda \o H_\mu).
\end{equation*}
Hence the $\gl_\infty$-module $\glV_{\lambda; 0} \o \glV_{0;\mu}$ is
realized as the direct summand  $(\SS_\lambda \VV \o \SS_\mu \VV_*)$
of the $\gl_\infty$-module $\VV^{\o (p,q)}$. Therefore
\begin{equation*}
    \soc^{(r)} \glV_{\lambda;\mu} = \glV_{\lambda;\mu} \cap \soc^{(r)} \VV^{\o d}.
\end{equation*}
It is known that for any partitions $\lambda,\mu,\lambda',\mu'$ we
have $[\varGamma_{\lambda; 0}^{(n)} \o \varGamma_{0;\mu}^{(n)} :
\varGamma_{\lambda';\mu'}^{(n)}] = \sum_{\gamma}
\LR_{\lambda',\gamma}^\lambda \LR_{\mu',\gamma}^\mu$ provided $n$
is large enough, see e.g. \cite{HTW}. This yields
\begin{equation}\label{eq:multies}
    [\glV_{\lambda; 0} \o \glV_{0;\mu} : \glV_{\lambda';\mu'}] =
    \sum_{\gamma} \LR_{\lambda',\gamma}^\lambda \LR_{\mu',\gamma}^\mu,
\end{equation}
and combining \eqref{eq:multies} with the description of the socle filtration of
$\VV^{\o(p,q)}$, we obtain \eqref{eq:socle layers gl}. In
particular, $\soc (\glV_{\lambda; 0} \o \glV_{0;\mu}) \cong
\glV_{\lambda;\mu}$, and the simplicity of the socle implies the
indecomposability of the $\gl_\infty$-module $\glV_{\lambda; 0} \o
\glV_{0;\mu}$.

\end{proof}

\begin{cor}
The decomposition of $\VV^{\o (p,q)}$ into indecomposable
$\gl_\infty$-modules is given by the isomorphism
\begin{equation*}
    \VV^{\o (p,q)} \cong \bigoplus_{|\lambda| = p} \bigoplus_{|\mu| = q}
    \( \dim H_\lambda \, \dim H_\mu \) \ \glV_{\lambda;0} \o \glV_{0;\mu}.
\end{equation*}
\end{cor}

\noindent {\bf Examples.} We begin with the description of tensors of rank 2.
Purely covariant and purely contravariant tensor representations of $\gl_\infty$ and $\sl_\infty$ are
completely reducible, and are decomposed by
the Schur-Weyl duality \eqref{eq:mixed Schur-Weyl
infinite}. For tensors of rank 2 these decompositions correspond to the
isomorphisms
\begin{equation*}
    \VV \o \VV \cong S^2 \VV \oplus \Lambda^2 \VV,
    \qquad\qquad
    \VV_* \o \VV_* \cong S^2 \VV_* \oplus \Lambda^2 \VV_*.
\end{equation*}
The tensor representation of mixed type $\VV \o \VV_*$ is the
adjoint representation of $\gl_\infty$, for which we have the
non-splitting short exact sequence of $\gl_\infty$-modules
\begin{equation*}
    0 \to \sl_\infty \to \gl_\infty \to \Bbbk \to 0.
\end{equation*}
In other words, the socle filtration of the adjoint
$\gl_\infty$-module has length 2, with a simple socle isomorphic to
$\sl_\infty$, and a simple top isomorphic to $\Bbbk$. We summarize
the structure of these modules graphically by
\begin{align*}
    \VV^{\o(2,0)}  &\sim
    \begin{tabular}{|c|}
        \hline
        $\glV_{(2);(0)}$\\
        \hline
    \end{tabular}
    \ \oplus \
    \begin{tabular}{|c|}
        \hline
        $\glV_{(1,1);(0)}$\\
        \hline
    \end{tabular}\ ,
    \\
    \VV^{(1,1)}  &\sim
    \begin{tabular}{|c|}
        \hline
        $\glV_{(0);(0)}$\\
        \hline
        $\glV_{(1);(1)}$\\
        \hline
    \end{tabular}\ ,
    \\
    \VV^{\o(0,2)}&\sim
    \begin{tabular}{|c|}
        \hline
        $\glV_{(0);(2)}$\\
        \hline
    \end{tabular}
    \ \oplus \
    \begin{tabular}{|c|}
        \hline
        $\glV_{(0);(1,1)}$\\
        \hline
    \end{tabular}\ .
\end{align*}
In these diagrams each tower represents an indecomposable direct summand of a module, and
the vertical arrangement of boxes in each tower represents the
structure of the layers of the socle filtration, with the bottom box
corresponding to the socle.

\medskip

\noindent Similarly, Theorem \ref{thm:indecomposables gl} yields the
following diagrams for tensor representations of rank 3:
\begin{align*}
    \VV^{\o(3,0)} &\sim
    \begin{tabular}{|c|}
        \hline
        $\glV_{(3);(0)}$\\
        \hline
    \end{tabular}
    \ \oplus 2\
    \begin{tabular}{|c|}
        \hline
        $\glV_{(2,1);(0)}$\\
        \hline
    \end{tabular}
    \ \oplus \
    \begin{tabular}{|c|}
        \hline
        $\glV_{(1,1,1);(0)}$\\
        \hline
    \end{tabular}\ ,
\\
    \VV^{\o(2,1)} & \sim
    \begin{tabular}{|c|}
        \hline
        $\glV_{(1);(0)}$\\
        \hline
        $\glV_{(2);(1)}$\\
        \hline
    \end{tabular}
    \ \oplus \
    \begin{tabular}{|c|}
        \hline
        $\glV_{(1);(0)}$\\
        \hline
        $\glV_{(1,1);(1)}$\\
        \hline
    \end{tabular}\ ,
\\
    \VV^{\o(0,3)} &\sim
    \begin{tabular}{|c|}
        \hline
        $\glV_{(0);(3)}$\\
        \hline
    \end{tabular}
    \ \oplus 2\
    \begin{tabular}{|c|}
        \hline
        $\glV_{(0);(2,1)}$\\
        \hline
    \end{tabular}
    \ \oplus \
    \begin{tabular}{|c|}
        \hline
        $\glV_{(0);(1,1,1)}$\\
        \hline
    \end{tabular} \ ,
\end{align*}
and for tensor representations of rank 4:
\begin{align*}
    \VV^{\o(4,0)} &\sim
    \begin{tabular}{|c|}
        \hline
        $\glV_{(4);(0)}$\\
        \hline
    \end{tabular}
    \ \oplus 3\
    \begin{tabular}{|c|}
        \hline
        $\glV_{(3,1);(0)}$\\
        \hline
    \end{tabular}
    \ \oplus 2\
    \begin{tabular}{|c|}
        \hline
        $\glV_{(2,2);(0)}$\\
        \hline
    \end{tabular}
    \ \oplus 3\
    \begin{tabular}{|c|}
        \hline
        $\glV_{(2,1,1);(0)}$\\
        \hline
    \end{tabular}
    \ \oplus \
    \begin{tabular}{|c|}
        \hline
        $\glV_{(1,1,1,1);(0)}$\\
        \hline
    \end{tabular}\ ,
\\
    \VV^{\o(3,1)} & \sim
    \begin{tabular}{|c|}
        \hline
        $\glV_{(2);(0)}$\\
        \hline
        $\glV_{(3);(1)}$\\
        \hline
    \end{tabular}
    \ \oplus 2 \
    \begin{tabular}{|c|}
        \hline
        $\glV_{(2);(0)} \oplus \glV_{(1,1);(0)}$\\
        \hline
        $\glV_{(2,1);(1)}$\\
        \hline
    \end{tabular}
    \ \oplus \
    \begin{tabular}{|c|}
        \hline
        $\glV_{(1,1);(0)}$\\
        \hline
        $\glV_{(1,1,1);(1)}$\\
        \hline
    \end{tabular}\ ,
\\
    \VV^{\o(2,2)} &\sim
    \begin{tabular}{|c|}
        \hline
        $\glV_{(0);(0)}$\\
        \hline
        $\glV_{(1);(1)}$\\
        \hline
        $\glV_{(2);(2)}$\\
        \hline
    \end{tabular}
    \ \oplus \
    \begin{tabular}{|c|}
        \multicolumn{1}{c}{ }\\
        \hline
        $\glV_{(1);(1)}$\\
        \hline
        $\glV_{(2);(1,1)}$\\
        \hline
    \end{tabular}
    \ \oplus \
    \begin{tabular}{|c|}
        \multicolumn{1}{c}{ }\\
        \hline
        $\glV_{(1);(1)}$\\
        \hline
        $\glV_{(1,1);(2)}$\\
        \hline
    \end{tabular}
    \ \oplus \
    \begin{tabular}{|c|}
        \hline
        $\glV_{(0);(0)}$\\
        \hline
        $\glV_{(1);(1)}$\\
        \hline
        $\glV_{(1,1);(1,1)}$\\
        \hline
    \end{tabular}\ ,
\\
    \VV^{\o(1,3)} & \sim
    \begin{tabular}{|c|}
        \hline
        $\glV_{(0);(2)}$\\
        \hline
        $\glV_{(1);(3)}$\\
        \hline
    \end{tabular}
    \ \oplus 2 \
    \begin{tabular}{|c|}
        \hline
        $\glV_{(0);(2)} \oplus \glV_{(0);(1,1)}$\\
        \hline
        $\glV_{(1);(2,1)}$\\
        \hline
    \end{tabular}
    \ \oplus \
    \begin{tabular}{|c|}
        \hline
        $\glV_{(0);(1,1)}$\\
        \hline
        $\glV_{(1);(1,1,1)}$\\
        \hline
    \end{tabular}\ ,
\\
    \VV^{\o(0,4)} &\sim
    \begin{tabular}{|c|}
        \hline
        $\glV_{(0);(4)}$\\
        \hline
    \end{tabular}
    \ \oplus 3\
    \begin{tabular}{|c|}
        \hline
        $\glV_{(0);(3,1)}$\\
        \hline
    \end{tabular}
    \ \oplus 2\
    \begin{tabular}{|c|}
        \hline
        $\glV_{(0);(2,2)}$\\
        \hline
    \end{tabular}
    \ \oplus 3\
    \begin{tabular}{|c|}
        \hline
        $\glV_{(0);(2,1,1)}$\\
        \hline
    \end{tabular}
    \ \oplus \
    \begin{tabular}{|c|}
        \hline
        $\glV_{(0);(1,1,1,1)}$\\
        \hline
    \end{tabular}\ .
\end{align*}

\section{Tensor representations of $\sp_\infty$}
\label{sec:sp}


Let $\VV$ be a countable dimensional vector space, and let $\Omega:
\VV \o \VV \to \Bbbk$ be a non-degenerate anti-symmetric bilinear
form on $\VV$. We realize the Lie algebra $\gl_\infty$ as in Section
\ref{sec:gl} by taking $\VV_* = \VV$ and $\Omega$ as the pairing $\<\cdot,\cdot\>$.
The Lie algebra $\sp_\infty$ is defined as the Lie subalgebra of
$\gl_\infty$, which preserves the form $\Omega$, i.e.
\begin{equation*}
    \sp_\infty = \left\{ X \in \gl_\infty \,\biggr|\, \Omega(X u,v) + \Omega(u,X v) = 0 \text{ for all } u,v \in \VV
    \right\}.
\end{equation*}
It is always possible to pick a basis $\{\xi_i\}$ of $\VV$, indexed
as before by the set $\mathfrak I = \Z\setminus\{0\}$, such that
$\Omega(\xi_i,\xi_j) = \operatorname{sign}(i) \, \delta_{i+j,0}$. In
the coordinate realization of $\gl_\infty$, the Lie algebra
$\sp_\infty$ has a linear basis $\{\operatorname{sign}(j) E_{i,j} -
\operatorname{sign}(i) E_{-j,-i}\}_{i,j\in \mathfrak I}$. Since the
dual basis $\{\xi_i^*\}_{i \in \mathfrak I}$ is given by $\xi_i^* =
\operatorname{sign}(i) \, \xi_{-i}$, it follows that $\sp_\infty =
S^2 \VV$, and the Lie bracket is induced by
\eqref{eq:def gl Lie bracket}.

We call $\VV$, regarded as an $\sp_\infty$-module by restriction,
the {\it natural representation} of $\sp_\infty$. It is easy to see that the $\sp_\infty$-action on the  $\gl_\infty$-module $\VV^{\o(p,q)}$ coincides with the $\sp_\infty$-action on the $\gl_\infty$-module $\VV^{\o(p+q)}$. Therefore it suffices to study the tensor representations $\VV^{\o d}$.

To define the notion of a highest weight $\sp_\infty$-module, we
consider the direct sum decomposition $\sp_\infty = \h_\sp
\oplus \(\bigoplus_{\a \in \Delta_{\sp}} \, \Bbbk \, X_\a^{\sp}\)$,
where
\begin{gather*}
    \h_\sp = \bigoplus_{i \in \mathfrak I} \ \Bbbk \, (E_{i,i} - E_{-i,-i}),
    \qquad
    \Delta_\sp = \{ \pm(\eps_i+\eps_j) \ | \ i,j \in \N  \} \bigcup
    \{ \eps_i-\eps_j \ | \ i,j \in \N \text{ and } i \ne j \},
    \\
    X_{\eps_i+\eps_j}^{\sp} = E_{i, -j} + E_{j, -i},
    \qquad\qquad
    X_{-\eps_i-\eps_j}^{\sp} = E_{-i,j} + E_{-j,i},
    \qquad\qquad
    X_{\eps_i-\eps_j}^{\sp} = E_{i,j} - E_{-j, -i}.
\end{gather*}
The subalgebra $\b_\sp$ and its ideal $\n_\sp$ are defined as
before using the polarization of the root system $\Delta_\sp =
\Delta_\sp^+ \coprod -\Delta_\sp^+$, where
\begin{equation*}
    \Delta_\sp^+ = \{ \eps_i+\eps_j \ | \ i,j \in \N \text{ and } i \le j \} \bigcup
    \{ \eps_i-\eps_j \ | \ i,j \in \N \text{ and } i < j \}.
\end{equation*}
An $\sp_\infty$-module $V$ is called a {\it highest weight module} with
highest weight $\chi \in \h_\sp^*$, if it is generated by a vector
$v\in V$ satisfying $\n_\sp \, v = 0$ and $H \, v = \chi(H)\, v$ for
all $H \in \h_\sp$. In particular, the natural representation $\VV$
is a highest weight $\sp_\infty$-module with highest weight $\eps_1$, generated
by the highest weight vector $\xi_1$.

\medskip

For any pair $I = (i,j)$ of integers such that $1 \le i < j
\le d$, the form $\Omega$ determines a contraction
\begin{equation*}
\begin{gathered}
    \Phi_{\<I\>}: \VV^{\o d} \to \VV^{\o (d-2)},\\
    v_1 \o \dots \o v_d \mapsto \Omega(v_{i},v_{j}) \, v_1 \o \dots \o \hat
    v_{i} \o \dots \o \hat v_{j} \o \dots \o v_d.
\end{gathered}
\end{equation*}
We set $\VV^{\<0\>}  \bydef \Bbbk$, $\VV^{\<1\>} \bydef \VV$ and
\begin{equation*}
    \VV^{\<d\>} \bydef \bigcap_{I} \ker
    \biggr( \Phi_{\<I\>}: \VV^{\o d} \to
    \VV^{\o(d-2)}\biggr)
\end{equation*}
for $d\ge
2$. For any partition $\lambda$ with $|\lambda|=d$ we define the
$\sp_\infty$-submodule $\spV_{\lambda}$ of $\VV^{\o d}$,
\begin{equation*}
    \spV_{\lambda} \bydef  \VV^{\<d\>} \cap \mathbb S_\lambda \VV.
\end{equation*}

\begin{thm}\label{thm:sp Schur-Weyl infinite}
For any nonnegative integer $d$ there is an isomorphism of
$(\sp_\infty,\mathfrak S_d)$-modules
\begin{equation}\label{eq:sp Schur-Weyl infinite}
    \VV^{\<d\>} \cong \bigoplus_{|\lambda| = d} \spV_\lambda \o
    H_\lambda.
\end{equation}
For every partition $\lambda$, the $\sp_\infty$-module
$\spV_{\lambda}$ is an irreducible highest weight module with
highest weight $\omega = \sum_{i \in \N} \lambda_i \, \eps_i$.
\end{thm}

\begin{proof}
For each $n$, denote by $V_{n}$ the $2n$-dimensional subspace of
$\VV$ spanned by $\xi_{\pm 1},\dots,\xi_{\pm n}$, and identify the
Lie subalgebra of $\End(V_n)$ which preserves the restriction of the form $\Omega$
to $V_{n}$ with the Lie algebra $\sp_{2n}$. The tensor representations
$(V_{n})^{\o d}$ of $\sp_{2n}$ and the contractions
$\Phi_{\<I\>}^{(n)}: (V_{n})^{\o d} \to (V_{n})^{\o(d-2)}$ are
defined as before. We set
\begin{equation*}
    (V_{n})^{\<d\>} \bydef \bigcap_{I} \ker \Phi_{\<I\>}^{(n)},
    \qquad\qquad
    \varGamma_{\<\lambda\>}^{(n)} \bydef  (V_{n})^{ \<d\> } \cap \SS_\lambda V_{n}.
\end{equation*}
The Weyl construction for $\sp_{2n}$ implies that there are
isomorphisms of $(\sp_{2n},\mathfrak S_d)$-modules
\begin{equation}\label{eq:sp Schur-Weyl finite}
    (V_n)^{\<d\>} \cong \bigoplus_{|\lambda|=d}
    \varGamma_{\<\lambda\>}^{(n)} \o H_\lambda,
\end{equation}
and that for $n\ge d$ the module $\varGamma_{\<\lambda\>}^{(n)}$ is
an irreducible highest weight module with highest weight $\sum_{i
\in \N} \lambda_i \, \eps_i$. As in the proof of Theorem
\ref{thm:mixed Schur-Weyl infinite}, the diagram of inclusions
\begin{equation}\label{eq:exhaustions sp}
\begin{gathered}\xymatrix@C=0pt@R=0pt{
    (V_d)^{\<d\>} & \subset & (V_{d+1})^{\<d\>} & \subset &\dots &\subset &(V_n)^{\<d\>}
    &\subset &\dots &\subset  &\bigcup_{n\in\N} (V_n)^{ \<d\> } & = & \VV^{\<d\>}
    \\
    \cup &&      \cup &&     &&     \cup &&     &&     \cup &&    \cup
    \\
    \varGamma_{\<\lambda\>}^{(d)} & \subset & \varGamma_{\<\lambda\>}^{(d+1)} &\subset &\dots &\subset &\varGamma_{\<\lambda\>}^{(n)}
    &\subset &\dots &\subset &\bigcup_{n\in\N} \varGamma_{\<\lambda\>}^{(n)} & = & \spV_{\lambda}
    }
\end{gathered}
\end{equation}
establishes both the irreducibility of the $\sp_\infty$-module
$\spV_{\lambda}$ and the isomorphism \eqref{eq:sp Schur-Weyl
infinite}, and the equalities
\begin{equation*}
    \varGamma_{\<\lambda\>}^{(d)}[\chi] = \varGamma_{\<\lambda\>}^{(d+1)}[\chi]=
    \dots = \varGamma_{\<\lambda\>}^{(n)}[\chi] = \dots = \bigcup_{n\in\N} \varGamma_{\<\lambda\>}^{(n)}[\chi] = \spV_{\lambda}[\chi]
\end{equation*}
yields the characterization of $\spV_{\lambda}[\chi]$ as a highest weight $\sp_\infty$-module.
\end{proof}


Next, we describe the socle filtration of the tensor representations
of $\sp_\infty$.
\begin{thm}
For any nonnegative integer \, $d$ \, the $\gl_\infty$-module
$\VV^{\o d}$, regarded as an $\sp_\infty$-module, has Loewy length
$\[\frac d2\]+1$, and
\begin{equation}\label{eq:socle series sp}
    \soc^{(r)} \VV^{\o d} = \bigcap_{I_1,\dots,I_r}
    \ker \biggr( \Phi_{\<I_1,\dots,I_r\>}: \VV^{\o d} \to \VV^{\o (d-2r)} \biggr),
    \qquad
    r = 1, \dots ,\[\frac d2\].
\end{equation}
\end{thm}
\begin{proof}
The proof is obtained as a minor variation of the proof of Theorem \ref{thm:socle series gl}.
Denoting by $\FF^{(r)}$ the module on the right side of \eqref{eq:socle series sp}, we argue as before that
\begin{equation*}
    \FF^{(r+1)}/\FF^{(r)} \cong \bigoplus_{|\lambda|=p-2r} \spV_{\lambda} \o H(\lambda;r)
\end{equation*}
for some $\mathfrak S_d$-modules $H(\lambda;r)$. The semisimplicity of these layers shows that
$\FF^{(r)} \subset \soc^{(r)} \VV^{\o d}$ for all $r$, and the opposite inclusion is proved by an obvious modification of Corollary \ref{thm:proval}.
\end{proof}

\begin{thm}\label{thm: socle sp}
For any partition $\lambda$ the $\sp_\infty$-module
$\glV_{\lambda;0}$ is indecomposable, and
\begin{equation}\label{eq:socle layers sp}
    \overline\soc^{(r+1)} \glV_{\lambda;0} = \bigoplus_{\mu}
    \( \sum_{|\gamma|=r} \LR_{\mu, \, (2\gamma)^\top}^{\lambda} \)\ \spV_{\mu},
    \qquad
    r = 1, \dots ,\[\frac d2\].
\end{equation}
\end{thm}

\begin{proof}
By construction, the $\gl_\infty$-module $\glV_{\lambda;0}$ is
realized as the direct summand $\mathbb S_{\lambda} \VV$ of the
$\gl_\infty$-module $\VV^{\o d}$. It remains a direct summand when
$\VV^{\o d}$ is regarded as an $\sp_\infty$-module, therefore $\soc^{(r)}
\glV_{\lambda;0} = \glV_{\lambda;0} \cap \soc^{(r)} \VV^{\o d}$.

It is known that for any partitions $\lambda,\mu$ we have
$[\varGamma_{\lambda}^{(2n)} : \varGamma_{\<\mu\>}^{(n)}] =
\sum_{\gamma} \LR_{\mu, \, (2\gamma)^\top}^{\lambda}$ provided $n$
is large enough, see e.g. \cite{HTW}. This implies that
$[\glV_{\lambda; 0} : \spV_{\mu}] = \sum_{\gamma} \LR_{\mu, \,
(2\gamma)^\top}^{\lambda}$, and combining it with the description of
the socle filtration of $\VV^{\o d}$, we get \eqref{eq:socle layers
sp}. In particular, $\soc \glV_{\lambda;0} \cong \spV_{\lambda}$,
and the simplicity of the socle implies the indecomposability of
$\glV_{\lambda;0}$ as an $\sp_\infty$-module.
\end{proof}


\noindent {\bf Examples.} We begin by describing the structure of
$\VV \o \VV = \Lambda^2 \VV \oplus S^2 \VV$ as an $\sp_\infty$-module. The symmetric square
$S^2 \VV = \glV_{(2);(0)}$ is the irreducible adjoint
$\sp_\infty$-module, isomorphic to $\spV_{(2)}$. For the exterior
square $\Lambda^2 \VV = \glV_{(1,1);(0)}$ one has the short exact
sequence of $\sp_\infty$-modules
\begin{equation*}
    0 \to \spV_{(1,1)} \to \Lambda^2 \VV \overset \Omega\to \Bbbk \to 0
\end{equation*}
which does not split. Therefore the structure of $\VV \o \VV$ is
graphically represented as
\begin{equation*}
    \VV\o \VV \quad\sim\quad
    \begin{tabular}{|c|}
        \multicolumn{1}{r}{ }\\
        \hline
        $\spV_{(2)}$\\
        \hline
    \end{tabular}
    \oplus
    \begin{tabular}{|c|}
        \hline
        $\spV_{(0)}$\\
        \hline
        $\spV_{(1,1)}$\\
        \hline
    \end{tabular}\, .
\end{equation*}

\bigskip

\noindent Similarly, the structure of the tensor representation of rank 3 is
represented as
\begin{equation*}
    \VV^{\o3} \quad\sim\quad
    \begin{tabular}{|c|}
        \multicolumn{1}{c}{ }\\
        \hline
        $\spV_{(3)}$\\
        \hline
    \end{tabular}
    \ \oplus 2 \
    \begin{tabular}{|c|}
        \hline
        $\spV_{(1)}$\\
        \hline
        $\spV_{(2,1)}$\\
        \hline
    \end{tabular}
    \ \oplus \
    \begin{tabular}{|c|}
        \hline
        $\spV_{(1)}$\\
        \hline
        $\spV_{(1,1,1)}$\\
        \hline
    \end{tabular}\ ,
\end{equation*}
and the structure of the tensor representation of rank 4 as
\begin{equation*}
    \VV^{\o4} \quad\sim\quad
        \begin{tabular}{|c|}
        \multicolumn{1}{c}{ }\\
        \multicolumn{1}{c}{ }\\
        \hline
        $\spV_{(4)}$\\
        \hline
    \end{tabular}
    \ \oplus 3 \
    \begin{tabular}{|c|}
        \multicolumn{1}{c}{ }\\
        \hline
        $\spV_{(2)}$\\
        \hline
        $\spV_{(3,1)}$\\
        \hline
    \end{tabular}
    \ \oplus 2 \
    \begin{tabular}{|c|}
        \hline
        $\spV_{(0)}$\\
        \hline
        $\spV_{(1,1)}$\\
        \hline
        $\spV_{(2,2)}$\\
        \hline
    \end{tabular}
    \ \oplus 3 \
    \begin{tabular}{|c|}
        \multicolumn{1}{c}{ }\\
        \hline
        $\spV_{(2)} \oplus \soV_{(1,1)}$\\
        \hline
        $\spV_{(2,1,1)}$\\
        \hline
    \end{tabular}
    \ \oplus \
    \begin{tabular}{|c|}
        \hline
        $\spV_{(0)}$\\
        \hline
        $\spV_{(1,1)}$\\
        \hline
        $\spV_{(1,1,1,1)}$\\
        \hline
    \end{tabular}\ .
\end{equation*}

\section{Tensor representations of $\so_\infty$}


Let $\VV$ be a countable dimensional vector space, and let $Q: \VV
\o \VV \to \Bbbk$ be a non-degenerate symmetric bilinear form on
$\VV$. We realize the Lie algebra $\gl_\infty$ as in Section
\ref{sec:gl} by taking $\VV_* = \VV$ and $Q$ as the pairing $\<\cdot,\cdot\>$. The Lie
algebra $\so_\infty^{(Q)}$ is defined as the Lie subalgebra of this
$\gl_\infty$, which preserves the form $Q$, i.e.
\begin{equation*}
    \so_\infty^{(Q)} = \left\{ X \in \gl_\infty \,\biggr|\, Q(X u,v) + Q(u,X v) = 0 \text{ for all } u,v \in \VV
    \right\}.
\end{equation*}
It is customary when dealing with the orthogonal groups to work over
algebraically closed fields. Throughout this section we assume
that $\Bbbk$ is algebraically closed. Then all non-degenerate forms
$Q$ on $\VV$ are equivalent (e.g. they can all be transformed to the
standard sum of squares), and as a consequence the corresponding
Lie algebras $\so_\infty^{(Q)}$ are isomorphic. Hence we drop $Q$
from the notation, and denote our Lie algebra simply by $\so_\infty$.

It is convenient to pick a basis $\{\xi_i\}$ of $\VV$, indexed as
before by the set $\mathfrak I = \Z\setminus\{0\}$, such that
$Q(\xi_i,\xi_j) = \delta_{i+j,0}$. In the coordinate realization
of $\gl_\infty$, the Lie algebra $\so_\infty$ has a linear basis
$\{E_{i,j} - E_{-j,-i}\}_{i,j\in \mathfrak I}$. Since the dual basis
$\{\xi_i^*\}_{i \in \mathfrak I}$ is given by $\xi_i^* = \xi_{-i}$,
it follows that $\so_\infty = \bigwedge^2 \VV$,
and the Lie bracket is induced by \eqref{eq:def gl Lie bracket}.

We call $\VV$, regarded as an $\so_\infty$-module by restriction,
the {\it natural representation} of $\so_\infty$.
It is easy to see that the $\so_\infty$-action on the  $\gl_\infty$-module $\VV^{\o(p,q)}$ coincides with the $\so_\infty$-action on the $\gl_\infty$-module $\VV^{\o(p+q)}$. Therefore it suffices to study the tensor representations $\VV^{\o d}$.

To define the notion of a highest weight $\so_\infty$-module, we
consider the direct sum decomposition $\so_\infty = \h_\so \oplus
\( \bigoplus_{\a \in \Delta_{\sp}} \, \Bbbk \, X_\a^{\so} \)$, where
\begin{gather*}
    \h_\so = \bigoplus_{i \in \mathfrak I} \ \Bbbk \, (E_{i,i} - E_{-i,-i}),
    \qquad\qquad
    \Delta_\so = \{ \pm \eps_i \pm \eps_j) \ | \ i,j \in \N \text{ and } i \ne j \},
    \\
    X_{\eps_i+\eps_j}^{\so} = E_{i,-j} - E_{j,-i}, \qquad
    X_{\eps_i-\eps_j}^{\so} = E_{i,j} - E_{ -j,  -i}, \qquad
    X_{-\eps_i-\eps_j}^{\so} = E_{-j,i} - E_{-i,j}.
\end{gather*}
The subalgebras $\b_\so$ and its ideal $\n_\so$ are defined as
before using the polarization of the root system $\Delta_\so =
\Delta_\so^+ \coprod -\Delta_\so^+$, where
\begin{equation*}
    \Delta_\so^+ = \{ \eps_i \pm \eps_j \ | \ i,j \in \N \text{ and } i < j \}.
\end{equation*}
An $\so_\infty$-module $V$ is called a {\it highest weight module} with
highest weight $\chi \in \h_\so^*$, if it is generated by a vector
$v\in V$ satisfying $\n_\so \, v = 0$ and $H \, v = \chi(H)\, v$ for
all $H \in \h_\so$. In particular, the natural representation $\VV$
is a highest weight module with highest weight $\eps_1$, generated
by a highest weight vector $\xi_1$.

\begin{rem}
Our choice of a splitting Cartan subalgebra $\h_\so$ leads to a
root decomposition of type $D_\infty$, corresponding to the "even"
infinite orthogonal series of Lie algebras $\so_{2n}$. It is also possible to
pick another splitting Cartan subalgebra $\tilde{\h}_\so$ which leads
to a root decomposition of type $B_\infty$, corresponding to the
"odd" infinite orthogonal series $\so_{2n+1}$. These two types of splitting Cartan subalgebras
are clearly not conjugated, and in \cite{DPS} it is proved that any splitting Cartan subalgebra of $\so_\infty$ is either "even" or "odd".
\end{rem}

\medskip

For any pair $I = (i,j)$ of integers, satisfying $1 \le i < j \le
d$, define the contraction
\begin{equation}
\begin{gathered}
    \Phi_{[I]}: \VV^{\o d} \to \VV^{\o (d-2)},\\
    v_1 \o \dots \o v_d \mapsto Q(v_i,v_j) \, v_1 \o \dots \o \hat
    v_i \o \dots \o \hat v_j \o \dots \o v_d,
\end{gathered}
\end{equation}
We set $\VV^{[0]}  \bydef \Bbbk$, $\VV^{[1]} \bydef \VV$ and
\begin{equation}
    \VV^{[d]} \bydef \bigcap_{I} \ker
    \biggr( \Phi_{[I]}: \VV^{\o d} \to
    \VV^{\o(d-2)}\biggr)
\end{equation}
for $d\ge2$. For any partition $\lambda$, define the $\so_\infty$-submodule
$\soV_{\lambda}$ of $\VV^{\o d}$,
\begin{equation*}
    \soV_{\lambda} \bydef  \VV^{[d]} \cap \mathbb S_\lambda \VV.
\end{equation*}

\begin{thm}\label{thm:so Schur-Weyl infinite}
For any $d \in \N$ there is an isomorphism of $(\so_\infty,\mathfrak
S_d)$-modules
\begin{equation}\label{eq:so Schur-Weyl infinite}
    \VV^{[d]} \cong \bigoplus_{|\lambda| = d} \soV_\lambda \o
    H_\lambda.
\end{equation}
For every partition $\lambda$, the $\so_\infty$-module
$\soV_{\lambda}$ is an irreducible highest weight module with
highest weight $\omega = \sum_{i \in \N} \lambda_i \, \eps_i$.
\end{thm}

\begin{proof}
For each $n$ denote by $V_{n}$ the $2n$-dimensional subspace of
$\VV$ spanned by $\xi_{\pm 1},\dots,\xi_{\pm n}$. Let $\g_n$ be the Lie subalgebra of $\End(V_n)$ which preserves the restriction
of $Q$ to $V_{n}$. It is clear that $\g_n \cong \so_{2n}$. The
tensor representations $(V_{n})^{\o d}$ of $\g_n$ and the
contractions $\Phi_{[I]}^{(n)}: (V_{n})^{\o d)} \to
(V_{n})^{\o(d-2)}$ are defined as before. We set
\begin{equation*}
    (V_{n})^{[d]} \bydef \bigcap_{I} \ker \Phi_{[I]}^{(n)},
    \qquad\qquad
    \varGamma_{[\lambda]}^{(n)} \bydef  (V_{n})^{ [d] } \cap \SS_\lambda V_{n}.
\end{equation*}
The Weyl construction for $\so_{2n}$ implies that there are isomorphisms of $(\so_{2n},\mathfrak S_d)$-modules
\begin{equation}\label{eq:so Schur-Weyl finite}
    (V_n)^{[d]} \cong \bigoplus_{|\lambda|=d}
    \varGamma_{[\lambda]}^{(n)} \o H_\lambda,
\end{equation}
and that for $n\ge d$ the module $\varGamma_{[\lambda]}^{(n)}$ is
an irreducible highest weight module with highest weight $\sum_{i
\in \N} \lambda_i \, \eps_i$. As in the proof of Theorem
\ref{thm:mixed Schur-Weyl infinite}, the diagram of inclusions
\begin{equation}\label{eq:exhaustions so}
\begin{gathered}\xymatrix@C=0pt@R=0pt{
    (V_1)^{[d]} & \subset & (V_2)^{[d]} & \subset &\dots &\subset &(V_n)^{[d]}
    &\subset &\dots &\subset  &\bigcup_{n\in\N} (V_n)^{ [d] } & = & \VV^{[d]}
    \\
    \cup &&      \cup &&     &&     \cup &&     &&     \cup &&    \cup
    \\
    \varGamma_{[\lambda]}^{(1)} & \subset & \varGamma_{[\lambda]}^{(2)} &\subset &\dots &\subset &\varGamma_{[\lambda]}^{(n)}
    &\subset &\dots &\subset &\bigcup_{n\in\N} \varGamma_{[\lambda]}^{(n)} & = & \soV_{\lambda}
    }
\end{gathered}
\end{equation}
establishes both the irreducibility of the $\so_\infty$-module
$\spV_{\lambda}$ and the isomorphism \eqref{eq:so Schur-Weyl
infinite}, and the equalities
\begin{equation*}
    \varGamma_{\<\lambda\>}^{(d)}[\chi] = \varGamma_{\<\lambda\>}^{(d+1)}[\chi]=
    \dots = \varGamma_{\<\lambda\>}^{(n)}[\chi] = \dots = \bigcup_{n\in\N} \varGamma_{\<\lambda\>}^{(n)}[\chi] = \spV_{\lambda}[\chi]
\end{equation*}
yield the characterization of $\soV_{\lambda}[\chi]$ as a highest weight $\so_\infty$-module.
\end{proof}


\bigskip

\begin{thm}
For any nonnegative integer $d$ the Loewy length of $\VV^{\o d}$,
regarded as an $\so_\infty$-module, equals $\[\frac d2\]+1$, and
\begin{equation}\label{eq:socle series so}
    \soc^{(r)} \VV^{\o d} = \bigcap_{I_1,\dots,I_r}
    \ker \biggr( \Phi_{[I_1,\dots,I_r]}: \VV^{\o d} \to \VV^{\o (d-2r)} \biggr),
    \qquad
    r = 1,\dots ,\[\frac d2\].
\end{equation}
\end{thm}
\begin{proof}
The proof is a minor variation of the proof of Theorem \ref{thm:socle series gl}.
Denoting by $\FF^{(r)}$ the module on the right side of \eqref{eq:socle series so}, we argue as before that
\begin{equation*}
    \FF^{(r+1)}/\FF^{(r)} \cong \bigoplus_{|\lambda|=p-2r} \soV_{\lambda} \o H(\lambda;r)
\end{equation*}
for some $\mathfrak S_d$-modules $H(\lambda;r)$. The semisimplicity of these layers shows that
$\FF^{(r)} \subset \soc^{(r)} \VV^{\o d}$ for all $r$, and the opposite inclusion is proved by an obvious modification of Corollary \ref{thm:proval}.
\end{proof}

\bigskip

\begin{thm}
Let $\lambda$ be a partition of a positive integer $d$. Then the
$\so_\infty$-module $\glV_{\lambda;0}$ is indecomposable, and the
layers of its socle filtration satisfy
\begin{equation}\label{eq:socle layers so}
    \overline\soc^{(r+1)} \glV_{\lambda;0} \cong
    \bigoplus_{\mu} \( \sum_{|\gamma|=r} \LR_{\mu, \, 2\gamma}^{\lambda} \)\
    \soV_{\mu}.
\end{equation}

\end{thm}

\begin{proof}
By construction, the $\gl_\infty$-module $\glV_{\lambda;0}$ is
realized as the direct summand $\mathbb S_{\lambda} \VV$ of the
$\gl_\infty$-module $\VV^{\o d}$. It remains a direct summand when
$\VV^{\o d}$ is regarded as an $\so_\infty$-module, therefore $\soc^{(r)}
\glV_{\lambda;0} = \glV_{\lambda;0} \cap \soc^{(r)} \VV^{\o d}$.

It is known that
$[\varGamma_{\lambda}^{(2n)} : \varGamma_{[\mu]}^{(n)}] =
\sum_{\gamma} \LR_{\mu, \, 2\gamma}^{\lambda}$ for any partitions $\lambda,\mu$, provided $n$ is
large enough, see e.g. \cite{HTW}. Hence
$[\glV_{\lambda; 0} : \soV_{\mu}] = \sum_{\gamma} \LR_{\mu, \,
2\gamma}^{\lambda}$, and combining this with the description of the
socle filtration of $\VV^{\o d}$, we get \eqref{eq:socle layers so}.
In particular, $\soc \glV_{\lambda;0} \cong \soV_{\lambda}$, and
simplicity of the socle implies the indecomposability of
$\glV_{\lambda;0}$ as an $\so_\infty$-module.
\end{proof}

\begin{cor}
The decomposition of $\VV^{\o d}$ into indecomposable
$\so_\infty$-modules is given by the isomorphism
\begin{equation*}
    \VV^{\o d} \cong \bigoplus_{|\lambda| = d} (\dim H_\lambda) \, \glV_{\lambda;0}.
\end{equation*}
\end{cor}

\bigskip

\noindent {\bf Examples.}
We begin by describing the structure of $\VV \o \VV = \Lambda^2 \VV
\oplus S^2 \VV$ as an $\so_\infty$-module. The exterior square $\Lambda^2 \VV =
\glV_{(1,1);(0)}$ is the irreducible adjoint $\so_\infty$-module, isomorphic to $\soV_{(1,1)}$. For the symmetric square $S^2 \VV =
\glV_{(2);(0)}$ one has the short exact sequence of
$\so_\infty$-modules
\begin{equation*}
    0 \to \soV_{(2)} \to S^2 \VV \overset Q\to \Bbbk \to 0
\end{equation*}
which does not split. Therefore, the structure of $\VV \o \VV$ is
graphically represented as
\begin{equation*}
    \VV\o \VV \quad\sim\quad
    \begin{tabular}{|c|}
        \hline
        $\soV_{(0)}$\\
        \hline
        $\soV_{(2)}$\\
        \hline
    \end{tabular}
    \oplus
    \begin{tabular}{|c|}
        \multicolumn{1}{r}{ }\\
        \hline
        $\soV_{(1,1)}$\\
        \hline
    \end{tabular} \ .
\end{equation*}
\bigskip
Similarly, the structure of the tensor representation of rank 3 is
represented as
\begin{equation*}
    \VV^{\o3} \quad\sim\quad
    \begin{tabular}{|c|}
        \hline
        $\soV_{(1)}$\\
        \hline
        $\soV_{(3)}$\\
        \hline
    \end{tabular}
    \ \oplus 2 \
    \begin{tabular}{|c|}
        \hline
        $\soV_{(1)}$\\
        \hline
        $\soV_{(2,1)}$\\
        \hline
    \end{tabular}
    \ \oplus \
    \begin{tabular}{|c|}
        \multicolumn{1}{c}{ }\\
        \hline
        $\soV_{(1,1,1)}$\\
        \hline
    \end{tabular}\ ,
\end{equation*}
and the structure of the tensor representation of rank 4 as
\begin{equation*}
    \VV^{\o4} \quad\sim\quad
        \begin{tabular}{|c|}
        \hline
        $\soV_{(0)}$\\
        \hline
        $\soV_{(2)}$\\
        \hline
        $\soV_{(4)}$\\
        \hline
    \end{tabular}
    \ \oplus 3 \
    \begin{tabular}{|c|}
        \multicolumn{1}{c}{ }\\
        \hline
        $\soV_{(2)} \oplus \soV_{(1,1)}$\\
        \hline
        $\soV_{(3,1)}$\\
        \hline
    \end{tabular}
    \ \oplus 2 \
    \begin{tabular}{|c|}
        \hline
        $\soV_{(0)}$\\
        \hline
        $\soV_{(2)}$\\
        \hline
        $\soV_{(2,2)}$\\
        \hline
    \end{tabular}
    \ \oplus 3 \
    \begin{tabular}{|c|}
        \multicolumn{1}{c}{ }\\
        \hline
        $\soV_{(1,1)}$\\
        \hline
        $\soV_{(2,1,1)}$\\
        \hline
    \end{tabular}
    \ \oplus \
    \begin{tabular}{|c|}
        \multicolumn{1}{c}{ }\\
        \multicolumn{1}{c}{ }\\
        \hline
        $\soV_{(1,1,1,1)}$\\
        \hline
    \end{tabular}\ .
\end{equation*}

\section{Tensor representations of root-reductive Lie algebras}


It is known that over an algebraically closed field any infinite-dimensional simple
locally finite Lie algebra which admits a root decomposition is classical, i.e. isomorphic to
$\sl_\infty, \sp_\infty$ or $\so_\infty$, see \cite{PS,NS}. Here we discuss a
generalization of our results for $\sl_\infty,\so_\infty$ and $\sp_\infty$ to a more general
class of infinite-dimensional Lie algebras.

Let $\mathfrak k$ be one of the Lie algebras $\sl_\infty, \sp_\infty$ or $\so_\infty$, and let $\h_{\mathfrak k}$ denote its splitting Cartan subalgebra introduced in previous sections. Let $\VV$ denote the natural representation of $\mathfrak k$. Suppose that a Lie algebra $\g$ is the semidirect sum of $\mathfrak k$
and some Lie algebra $\mathfrak m$,
\begin{equation*}
    \g = \mathfrak k \cplus \mathfrak m,
\end{equation*}
and suppose furthermore that $\g$ has a subalgebra $\h$ such that
\begin{equation*}
    \h = \h_{\mathfrak k} \oplus \mathfrak m.
\end{equation*}

\begin{thm}\label{thm:generalization socle}
The socle filtration of the tensor representation of $\VV^{\o(p,q)}$ as a $\g$-module coincides with the socle  filtration of $\VV^{\o(p,q)}$ as a $\mathfrak k$-module.
\end{thm}

\begin{proof}
Since $\mathfrak m$ commutes with $\h_{\mathfrak k}$, the action of $\mathfrak m$ preserves the $\h_\mathfrak k$-weight subspaces of any $\mathfrak k$-module. The $\h_\mathfrak k$-weight subspaces of $\VV$ are one-dimensional, hence any $H \in \mathfrak m$ acts in any $\VV[\chi]$ as a scalar. Thus $\VV$ admits a $\h$-weight subspace decomposition, and the same is true for the weight subspaces of tensor representations $\VV^{\o(p,q)}$.

This shows that each $\mathfrak k$-submodule of $\VV^{\o(p,q)}$ is automatically a $\g$-module, and thus the socle filtration of $\VV^{\o(p,q)}$ as a $\mathfrak k$-module is a filtration by $\g$-submodules.
Moreover, the layers of the socle filtration for $\mathfrak k$ remain semisimple as $\g$-modules, and the statement follows.

\end{proof}

Our main application of Theorem \ref{thm:generalization socle} is to the class of
infinite-dimensional root-reductive Lie algebras, studied in
\cite{DP1,PS}. Recall the following structural theorem
from \cite{DP1}.

\begin{thm}\label{thm:structure root reductive}
Let $\g$ be a root reductive Lie algebra. Set $\mathfrak s =
[\g,\g]$ and $\mathfrak a = \mathfrak g/[\g,\g]$. Then
\begin{equation*}
    \mathfrak s = \bigoplus_{i \in \mathcal I} \mathfrak s^{(i)},
\end{equation*}
where each $\mathfrak s^{(i)}$ is isomorphic either to $\sl_\infty, \so_\infty,
\sp_\infty$, or to a simple finite-dimensional Lie algebra, and
$\mathcal I$ is an at most countable index set. Moreover, the short
exact sequence of Lie algebras
\begin{equation}\label{eq:ags exact sequence}
    0 \to \mathfrak s \to \g \to \mathfrak a \to 0
\end{equation}
splits. In other words, $\g \cong \mathfrak s \cplus \mathfrak a$.
\end{thm}

To apply Theorem \ref{thm:generalization socle} to a root-reductive Lie algebra $\g$, we identify $\g$ with a semidirect sum $\g \cong \mathfrak s \cplus \mathfrak a$, pick an infinite-dimensional direct summand $\mathfrak k = \mathfrak s^{(j)}$ of $\mathfrak s$, and set $\mathfrak m = \(\bigoplus_{i \ne j} \mathfrak s^{(i)}\) \cplus \mathfrak a$. The Lie algebra $\mathfrak g$ acts in the natural representation $\VV$ of $\mathfrak s^{(j)}$: $\mathfrak s^{(i)}$ annihilate $\VV$ for $i\ne j$, and $\mathfrak a$ acts by scalars in each $\h_\mathfrak k$-weight subspace.

In contrast with the finite-dimensional reductive Lie algebras,
there exist root-reductive $\g$, such that $\g \ncong \mathfrak s
\oplus \mathfrak a$. For example, for $\g = \gl_\infty$ one has
$\mathfrak s = \sl_\infty$ and $\mathfrak a = \Bbbk$, but
$\gl_\infty \ncong \sl_\infty \oplus \Bbbk$. Another interesting
example is the Lie algebra $\tilde\g$, constructed via the following root
injections
\begin{equation*}
    \gl_n \hookrightarrow \gl_{n+2},
    \qquad\qquad
    A \mapsto
    \begin{pmatrix} \frac{\Tr(A)}n & & \\ & A & \\ & & 0
    \end{pmatrix}.
\end{equation*}
Then $\tilde\g$ is not isomorphic to $\gl_\infty$, although it can
still be included in a short exact sequence of Lie algebras $0 \to
\sl_\infty \to \tilde\g \to \Bbbk \to 0$, see \cite{DPS}. However, Theorem \ref{thm:generalization socle} still applies and describes the socle filtration of the tensor representations of $\tilde\g$.

\section{Appendix}

For completeness, we discuss the details of Weyl's duality approach,
see \cite{W,FH}.

Let $p,q, n$ be nonnegative integers such that $n > p+q$.
Let $V_n = \Bbbk^n$ be the natural representation of the Lie algebra $\gl_n$.
For partitions $\lambda,\mu$ such that $|\lambda| = p$ and $|\mu| = q$
we denote by $\varGamma_{\lambda;\mu}^{(n)}$ the standard irreducible highest weight $\gl_n$-module with highest weight
$\omega = (\lambda_1,\dots,\lambda_p,0,\dots,0,-\mu_q,\dots,-\mu_1)$.

\begin{prop}
For any $n>p+q$ there is an isomorphism \eqref{eq:mixed Schur-Weyl finite}
\begin{equation*}
    (V_n)^{\{p,q\}} \cong \bigoplus_{|\lambda|=p} \bigoplus_{|\mu| = q}
    \Gamma_{\lambda;\mu}^{(n)} \o (H_\lambda \o H_\mu).
\end{equation*}

\end{prop}
\begin{proof}
For any $I=(i,j)$ with $i\in\{1,2,\dots,p\}$ and
$i\in\{1,2,\dots,q\}$, we define the inclusion
\begin{equation*}
\begin{gathered}
    \Psi_{I}^{(n)}: (V_n)^{\o (p-1,q-1)} \to (V_n)^{\o (p,q)},\\
    v_1 \o\dots\o v_{p-1} \o v^*_1 \o\dots\o v^*_{q-1} \mapsto \frac 1n \sum_{k=1}^n \,  \dots \o
    v_{i-1} \o \zeta_k \o v_{i+1} \o\dots\o v^*_{j-1} \o \zeta_k^* \o v_{j+1}^* \o
    \dots,
\end{gathered}
\end{equation*}
where $\{\zeta_k\}$ and $\{\zeta_k^*\}$ are any dual bases of $V_n$ and $V_n^*$ respectively. Set $\theta_I^{(n)} = \Psi_I^{(n)} \Phi_I^{(n)}$. The operators
$\theta_I^{(n)}$ are idempotent, and $(V_n)^{\{p,q\}} = \bigcap_{I}
\ker \theta_{I}^{(n)}$.

Let $\widetilde{\mathcal A}$ denote the subalgebra of endomorphisms
of $(V_n)^{\o(p,q)}$, generated by the images of elements of $\g_n$,
and let $\widetilde{\mathcal B}$ denote its commutator subalgebra,
i.e. the set of all endomorphisms of $(V_n)^{\o(p,q)}$, commuting
with the action of $\g_n$. From invariant theory it is known that
$\widetilde{\mathcal B}$ is generated by $\theta_I^{(n)}$ for
various $I$, and by permutation maps corresponding to elements from
$\mathfrak S_p \times \mathfrak S_q$. A general result from the
theory of semisimple finite-dimensional algebras implies that,
conversely, $\widetilde{\mathcal A}$ is the commutator subalgebra of
$\widetilde{\mathcal B}$. In other words, any endomorphism of
$(V_n)^{\o(p,q)}$, commuting with all $\theta_I^{(n)}$ and all
permutations from $\mathfrak S_p \times \mathfrak S_q$, must lie in
$\widetilde{\mathcal A}$.

Let $\mathcal A$ denote the subalgebra of endomorphisms of $(V_n)^{
\{p,q\} }$, generated by the images of elements of $\g_n$, and let
$\mathcal B$ denote the subalgebra of endomorphisms of $(V_n)^{
\{p,q\} }$, generated by the permutations from $\mathfrak S_p \times
\mathfrak S_q$. We claim that $\mathcal A$ and $\mathcal B$ are each
other's commutator subalgebras in $\End\( (V_n)^{ \{p,q\} } \)$.
Indeed, suppose that $L$ lies in the commutator subalgebra of
$\mathcal B$. Using the $(\g_n,\mathfrak S_p \times \mathfrak
S_q)$-module isomorphism (see e.g. \cite{FH} for details)
\begin{equation*}
    (V_n)^{\o (p, q)} = (V_n)^{ \{p,q\} } \oplus
    \sum_{I} \im \, \theta_{I}^{(n)},
\end{equation*}
we construct an endomorphism $\widetilde L$ of $(V_n)^{\o (p, q)}$
by extending $L$ trivially on the second direct summand. It is clear
that $\widetilde L$ commutes with all permutations from $\mathfrak
S_p \times \mathfrak S_q$, and also with the operators $\theta_I^{(n)}$, all of which
act on $(V_n)^{ \{p,q\} }$ by zero. Hence $\widetilde L$ belongs to
the commutator subalgebra of $\widetilde{\mathcal B}$, i.e.
$\widetilde L \in \widetilde{\mathcal A}$, and by restriction $L \in
\mathcal A$. Thus $\mathcal A$ is the commutator subalgebra of
$\mathcal B$, and it follows that $\mathcal B$ is also the
commutator subalgebra of $\mathcal A$.

The general theory of dual pairs and the fact that $\{H_\lambda \o
H_\mu\}_{|\lambda|=p, |\mu|=q}$ is a complete list of irreducible
$\mathfrak S_p \times \mathfrak S_q$-modules imply the existence of an isomorphism
\begin{equation*}
    (V_n)^{\{p,q\}} \cong \bigoplus_{|\lambda|=p} \bigoplus_{|\mu| = q}
    \Gamma(\lambda,\mu) \o (H_\lambda \o H_\mu)
\end{equation*}
for some irreducible $\gl_n$-modules $\Gamma(\lambda,\mu)$. To identify these modules explicitly, we note that Schur-Weyl duality yields
\begin{equation*}
    (V_n)^{\o p} \cong \sum_{|\lambda|=p} \varGamma_{\lambda;0}^{(n)} \o H_\lambda,
    \qquad\qquad
    (V_n^*)^{\o q} \cong \sum_{|\mu|=q} \varGamma_{0;\mu}^{(n)} \o H_\mu,
\end{equation*}
and therefore $\Gamma(\lambda,\mu)$ must be a submodule of $\varGamma_{\lambda;0}^{(n)} \o \varGamma_{0;\mu}^{(n)}$. On the other hand, the submodule
$\varGamma_{\lambda;\mu}^{(n)}$ of this tensor product does not occur as a submodule of $(V_n)^{\o(p-1,q-1)}$,
and thus lies in the kernel of all operators $\Phi_I$. We conclude that $\varGamma_{\lambda;\mu}^{(n)} \subset \Gamma(\lambda,n)$, and the irreducibility of
$\Gamma(\lambda,\mu)$ yields the desired statement.
\end{proof}

\bigskip

Finally, to prove the technical statement used in the proof of Theorem \ref{thm:socle series gl}, we need a preparatory lemma.

Define the contractions $\Phi_{I_1,\dots,I_r}: \VV^{\o(p,q)} \to \VV^{\o(p-r,q-r)}$ as the $r$-fold
convolutions between copies of $\VV$ and $\VV_*$ indicated by the
pairwise disjoint collection of index pairs $I_1,\dots,I_r$.
For $r=1,\dots,\ell$, define the inclusions
$\Psi_{I_1,\dots,I_r}^{(n)}: (V_n)^{\o(p-r,q-r)} \to
(V_n)^{\o(p,q)}$, by analogy with $\Phi_{I_1,\dots,I_r}$, as the
$r$-fold insertions of the canonical element of $V_n \o V_n^*$ into
positions specified by disjoint pair of indices $I_1,\dots,I_r$. Set
also
\begin{equation*}
    (V_n)^{ \{p,q\} }_r = \sum_{I_1,\dots,I_r} \im \(
    \Psi_{I_1,\dots,I_r}^{(n)} :(V_n)^{\o(p-r,q-r)} \to
    (V_n)^{\o(p,q)} \).
\end{equation*}
It is a standard exercise to show that for each $n$ one has the
direct sum decomposition
\begin{equation*}
    (V_n)^{\o (p, q)} = (V_n)^{ \{p,q\} }_0 \oplus  (V_n)^{ \{p,q\} }_1 \oplus (V_n)^{ \{p,q\} }_2
    \oplus \dots \oplus (V_n)^{ \{p,q\} }_\ell.
\end{equation*}
For any $I = (i,j)$ we consider the linear map
\begin{equation*}
\begin{gathered}
    \Xi^{(n)}_I: (V_n)^{\o (p,q)} \to (V_n)^{\o (p-1,q-1)},\\
    v_1 \o\dots\o v_p \o v^*_1 \o\dots\o v^*_q \mapsto \
    n\, \<\xi_n^*,v_i\> \<v_j^*,\xi_n\> \ v_1 \o \dots \o \hat v_i \o \dots \o v_p \o v^*_1 \o\dots \o \hat v_j^* \o \dots \o v_q^*.
\end{gathered}
\end{equation*}
\begin{lem}\label{thm:asymptotics lemma}
For any $v \in \VV^{\{p,q\}}$ we have
\begin{equation*}
    \lim_{n\to \infty} \Xi^{(n)}_{J_1} \, \Phi_{J_2,\dots,J_r}^{(n)}
    \Psi_{I_1,\dots,I_r}^{(n)} v =
    \begin{cases}
        v, & \text{if $\{I_1,\dots,I_r\} = \{J_1,\dots,J_r\}$ as
        sets}\\
        0, & \text{ otherwise}
    \end{cases}.
\end{equation*}
\end{lem}
\begin{proof}
We use induction on $s$. The base of induction $s=1$ states that
\begin{equation*}
    \lim_{n\to \infty} \Xi^{(n)}_{J} \, \Psi_{I}^{(n)} v =
    \begin{cases}
        v, & \text{if $I = J$}\\
        0, & \text{ otherwise}
    \end{cases},
\end{equation*}
which is clear from the definition. Assume now that $r\ge2$, and let $J_r = (i,j)$.

Case 1. There exists $k$ such that $I_k = (i,j)$; we may assume that $k=r$. Then
\begin{equation*}
    \Xi^{(n)}_{J_1} \, \Phi_{J_2,\dots,J_r}^{(n)}
    \Psi_{I_1,\dots,I_r}^{(n)} v =
    \Xi^{(n)}_{J_1} \, \Phi_{J_2,\dots,J_{r-1}}^{(n)}
    \Psi_{I_1,\dots,I_{r-1}}^{(n)} \Phi_{I_r} \Psi_{J_r} v =
    \Xi^{(n)}_{J_1} \, \Phi_{J_2,\dots,J_{r-1}}^{(n)}
    \Psi_{I_1,\dots,I_{r-1}}^{(n)} v,
\end{equation*}
and the desired statement follows from the induction hypothesis.

Case 2. There exist $k,l$ such that $I_k = (i,a)$ and $I_l = (b,j)$; we may assume that $k=r$ and $l=r-1$. Setting $I' = (b,a)$ and using the identity $\Phi_{(i,j)}^{(n)} \Psi_{(i,a),(b,j)}^{(n)} = \frac 1n \Psi_{b,a}^{(n)}$, we get
\begin{equation*}
    \Xi^{(n)}_{J_1} \, \Phi_{J_2,\dots,J_r}^{(n)}
    \Psi_{I_1,\dots,I_r}^{(n)} v =
    \Xi^{(n)}_{J_1} \, \Phi_{J_2,\dots,J_{r-1}}^{(n)}
    \Psi_{I_1,\dots,I_{r-2}}^{(n)} \Phi_{J_r}^{(n)} \Psi_{I_{r-1},I_r}^{(n)} v =
    \frac 1n \, \Xi^{(n)}_{J_1} \, \Phi_{J_2,\dots,J_{r-1}}^{(n)}
    \Psi_{I_1,\dots,I_{r-2},I'}^{(n)} v.
\end{equation*}
Applying the induction hypothesis, in both cases we obtain
$\lim_{n\to \infty} \Xi^{(n)}_{J_1} \, \Phi_{J_2,\dots,J_r}^{(n)}
    \Psi_{I_1,\dots,I_r}^{(n)} v =0$.

Case 3. There exists $k$ such that $I_k = (i,a)$, but $j$ never occurs in the second position of any $I_l$; we may assume that $k=r$. Using the identity $\Phi_{(i,j)}^{(n)} \Psi_{(i,a)}^{(n)} v = \frac mn \Phi_{(i,j)}^{(m)} \Psi_{(i,a)}^{(m)}v$, we get
\begin{multline*}
    \lim_{n\to \infty} \Xi^{(n)}_{J_1} \, \Phi_{J_2,\dots,J_r}^{(n)}
    \Psi_{I_1,\dots,I_r}^{(n)} v =
    \lim_{n\to \infty} \Xi^{(n)}_{J_1} \, \Phi_{J_2,\dots,J_{r-1}}^{(n)}
    \Psi_{I_1,\dots,I_{r-1}}^{(n)} \Phi_{J_r}^{(n)} \Psi_{I_r}^{(n)} v
    \\
    =  \lim_{n\to \infty} \frac mn \ \Xi^{(n)}_{J_1} \, \Phi_{J_2,\dots,J_{r-1}}^{(n)}
    \Psi_{I_1,\dots,I_{r-1}}^{(n)} \( \Phi_{J_r}^{(m)} \Psi_{I_r}^{(m)} v \).
\end{multline*}
Applying the induction hypothesis to $\Phi_{J_r}^{(m)} \Psi_{I_r}^{(m)} v \in \VV^{\{p,q\}}$, we see that the desired statement holds.

Case 4. The index $i$ never occurs in the first position of any $I_k$, and $j$ never occurs in the second position of any $I_l$. Then for all $n$ we have
\begin{equation*}
    \Xi^{(n)}_{J_1} \, \Phi_{J_2,\dots,J_r}^{(n)}
    \Psi_{I_1,\dots,I_r}^{(n)} v =
    \Xi^{(n)}_{J_1} \, \Phi_{J_2,\dots,J_{r-1}}^{(n)}
    \Psi_{I_1,\dots,I_r}^{(n)} \Phi_{I_r} v = 0.
\end{equation*}

\end{proof}

We are now ready to prove the following assertion, used in the proof of Theorem \ref{thm:socle series gl}. 
\begin{prop}\label{thm:proval}
There exist infinitely many $n$ such that $u - \pi_n(u) \notin F_n^{(s-1)}$.
\end{prop}
\begin{proof}
Assume that, on the contrary, $u - \pi_n(u) \in F_n^{(s-1)}$ for all $n\gg m$. Let $J_1,\dots,J_s$ be any collection of pairwise disjoint indices. Since $\Phi_{J_2,\dots,J_r} (u - \pi_n(u)) = 0$, we obtain for all $n$
\begin{equation*}
    \Xi^{(n)}_{J_1} \, \Phi_{J_2,\dots,J_r} \, \pi_n(u) =
    \Xi^{(n)}_{J_1} \, \Phi_{J_2,\dots,J_r} \, u = 0.
\end{equation*}
On the other hand, the vector $\pi_n(u)$ can be represented as
\begin{equation*}
    \pi_n(u) = \sum_{I_1,\dots,I_s} \Psi_{I_1,\dots,I_s} \, \zeta_{I_1,\dots,I_s}
\end{equation*}
for some collection $\{\zeta_{I_1,\dots,I_s} \}$ of vectors from $(V_n)^{\{p,q\}}$, and according to
Lemma \ref{thm:asymptotics lemma}
\begin{equation*}
    \lim_{n\to \infty} \Xi^{(n)}_{J_1} \, \Phi_{J_2,\dots,J_r} \, \pi_n(u) = \zeta_{J_1,\dots,J_n}.
\end{equation*}
It follows that $\zeta_{J_1,\dots,J_n} = 0$, and thus $\pi_n(u) = 0$. This contradicts the assumption that $\pi_n(u)$ generates a submodule of $F_n^{(s+1)}/F_n^{(r)}$ isomorphic to $\varGamma_{\lambda;\mu}^{(n)}$.
\end{proof}

\end{document}